\documentclass[reqno,11pt]{amsart}
\usepackage{amssymb}
\usepackage{mathrsfs}
\usepackage{amsmath,amssymb}
\usepackage{paralist}
\usepackage{graphics} %% add this and next lines if pictures should be in esp format
\usepackage{epsfig} %For pictures: screened artwork should be set up with an 85 or 100 line screen
\usepackage[colorlinks=true]{hyperref}
\hypersetup{urlcolor=red, citecolor=blue}
\usepackage[top=1in,bottom=1in,left=1in,right=1in]{geometry}
\usepackage{cite}
\newtheorem{thm}{Theorem}[section]

\newtheorem{lem}[thm]{Lemma}
\newtheorem{prop}[thm]{Proposition}
\theoremstyle{definition}

\numberwithin{equation}{section}
\newcommand{\be}{\begin{equation}}
\newcommand{\ee}{\end{equation}}
\begin{document}
\title{Minimizers of nonlocal interaction functional with
exogenous potential}
\author[Wang]{Wanwan Wang}%
\address{School of Mathematics, Southeast University, Nanjing 211189, P. R. China}
\email{wwwang2014@yeah.net}
\author[Li]{Yuxiang Li}%
\address{School of Mathematics, Southeast University, Nanjing 211189, P. R. China}
\email{lieyx@seu.edu.cn}

\begin{abstract}
The purpose of this paper is to consider the minimization problem of the following nonlocal interaction functional
\begin{equation*}
E[\rho]=\frac{1}{2}\int_{\mathbb{R}^N} \int_{\mathbb{R}^N}K(x-y)\rho(x)\rho(y)dxdy+\int_{\mathbb{R}^N}F(x)\rho(x)dx.
\end{equation*}
The kernel $K(x)=\frac{1}{q}|x|^q-\frac{1}{p}|x|^p$ is an endogenous potential,
where $q>p>-N$. The exogenous potential $F$ is a nonnegative continuous function and satisfies
$F(x)\to +\infty$ as $|x|\to +\infty$.
The existence of minimizers  are established based on the
concentration compactness principle.
Especially, for $F(x)=\beta|x|^2(\beta >0)$
and $K(x)=\frac{1}{2}|x|^2-\frac{1}{2-N}|x|^{2-N}$($N>2$),
the global minimizer
is given explicitly by the method of calculus of variation.
\end{abstract}

\medskip

\maketitle

\section{Introduction}
In this paper, we consider the minimization problem of the following nonlocal interaction functional
\begin{equation}\label{eq 1}
E[\rho]=\frac{1}{2}\int_{\mathbb{R}^N} \int_{\mathbb{R}^N}K(x-y)\rho(x)\rho(y)dxdy+\int_{\mathbb{R}^N}F(x)\rho(x)dx,
\end{equation}
where integer $N\geq 1$ and  the population density $\rho\geq 0$.
The exogenous potential $F$ is a nonnegative continuous function and satisfies $F(x)\to +\infty$
as $|x|\to +\infty$.
The kernel
\begin {equation}\label{eq d1}
K(x)=\frac{1}{q}|x|^q-\frac{1}{p}|x|^p
\end {equation}
is power-law repulsive-attractive potential with $q>p>-N$.

\vskip 2mm
It is well known that the interaction functional $E[\rho]$
is related to a class of biological aggregation models \cite{CT,CFT,BCLR,CHM,SST}.
The aggregation models consist of the following continuity equation in $\mathbb{R}^N$
\begin{equation}\label{eq a2}
\rho_t+\nabla\cdot(\rho V)=0,\ \ \ \ V=-\nabla K*\rho-\nabla F,
\end{equation}
where the velocity field is denoted by $V$.
The functions $K:\mathbb{R}^N\rightarrow \mathbb{R}$ and
$F:\mathbb{R}^N\rightarrow \mathbb{R}$ represent the endogenous potential and exogenous potential respectively.
The convolution kernel $-\nabla K$ incorporates the
endogenous forces arising from the inter-individual (attraction and repulsion) interactions,
see \cite{BT} for instance.
In fact, the equation (\ref{eq a2}) can be considered as a gradient flow of the functional (\ref{eq 1})
with respect to the
Euclidean Wasserstein metric \cite{AGS,CDFLS,CMV,V}.
This equation appears in various contexts, including animals
in flock patterns in biological swarms \cite{ME,TBL}, robotic swarming \cite{CHDB,PEG},
granular media \cite{BCP,BCCP,CMV,T},
self-assembly of nanoparticles \cite{HP,HPP},
Ginzburg-Landau vortices \cite{DZ,MZ,W}, etc.

\vskip 2mm
The minimization problem of the following nonlocal functional without exogenous potential
\begin{equation}\label{eq 2}
E[\rho]=\frac{1}{2}\int_{\mathbb{R}^N} \int_{\mathbb{R}^N}K(x-y)\rho(x)\rho(y)dxdy
\end{equation}
with $K$ in (\ref {eq d1}) has been extensively studied in the past few years.
Balagu\'{e} et al. \cite{BCLR} proved that the more repulsive the potential $K$ is at the origin, the higher dimension
of the support of local minimizers will be.
In \cite{BCLRn} the same authors gave conditions for radial stability/instability of particular local minimizers.
Carrillo et al. \cite{CCH} showed the existence of global minimizers in the discrete setting for $q>p$ ,
and obtained the uniform bound on the diameter of global minimizers
for $q>p\geq 1$.
Especially, for the one-dimensional case, the discrete minimizer is unique
and symmetric with respect to its centre of mass for $p\leq 1$, $q\geq 1$ and $q>p$.
Choksi et al. \cite{CFT} showed the existence of minimizers in the class of probability measures
when the power of repulsion $p$  is positive.
For the repulsion having a singularity at the origin, i.e. $p<0$,
they established the existence of minimizers in a  class of  bounded
$L^1$-functions satisfying a given mass constraint.
In the special case of Newtonian repulsion
and quadratic attraction, they showed that the unique global minimizer
is the characteristic function on a ball. We refer to \cite{FH,FHK} for further reading.
\vskip 2mm

The minimization problem of the nonlocal functional (\ref {eq 2}) with general potential $K$
has also been studied in some literatures.
 When $K$ is supposed to be lower-semicontinuous and locally integrable on $\mathbb{R}^N$,
Simione et al.  \cite{SST} obtained that the functional admits a global minimizer
if $K(x)\to +\infty$ as $|x|\to +\infty$,
and it also admits  a global minimizer if $K(x)\to 0$ as $|x|\to +\infty$ and
some additional assumptions of the functional.
For $K$ lower-semicontinuous and locally integrable on $\mathbb{R}^N$,
Carrillo et al. \cite{CDM} also  proved that if this repulsion is like Newtonian or more singular than Newtonian,
 the local minimizers must be locally bounded.
Moreover, under some suitable assumption on the potential $K$,
Ca\~{n}izo et al. \cite{CCP} proved that there exist global minimizers which are compact.
Carrillo et al. \cite{CFP} showed that the support of any local minimizer consists of isolated points
whenever the interaction potential is of class $C^2$ and mildly repulsive at the origin.
We  refer to  \cite{CDFLS,CMV,BL,BV,LTB} for further reading.

\vskip 2mm
For positive exogenous potential, the functional (\ref{eq 1}) is studied by some authors.
In one spatial dimension, and for several choices of endogenous potentials,
external forces, Bernoff and Topaz \cite{BT} found exact analytical expressions
of the minimizer of the  population density.
For  $K(x)=e^{-|x|}$ and $F(x)=\beta x^2$ ($\beta>0$), for given mass $m$, they found
the minimizer of the following form
\begin{eqnarray*}
\rho_0(x)=\left\{\begin{array}{ll}{\frac{\beta}{2}\left[\left(\frac{3 m}
{2 \beta}+1\right)^{\frac{2}{3}}+\left(1-x^{2}\right)\right],} & {|x|
\leq\left(\frac{3 m}{2 \beta}+1\right)^{\frac{1}{3}}-1}, \\ {0,}
& {|x|>\left(\frac{3 m}{2 \beta}+1\right)^{\frac{1}{3}}-1}\end{array}\right.
\end{eqnarray*}
satisfying  $\int_\mathbb{R}\rho_0(x)dx=m$.
For $F(x) = gx$ ($x>0$) with $g<m$, they found the minimizer of the following form
\begin{eqnarray*}
\rho_0(x)=\left\{\begin{array}{ll}{\sqrt{g m}-\frac{g}{2}(1+x)+\sqrt{g m} \delta(x),}
& {0 \leq x \leq 2 \sqrt{m / g}-2}, \\ {0,} & {x>2 \sqrt{m / g}-2}.\end{array}\right.
\end{eqnarray*}
The exact solutions provide a sampling of the wide variety of equilibrium configurations
possible within the general swarm modeling framework.
Some other authors also obtained the asymptotic behaviour for (\ref {eq a2}) with exogenous potential $F$,
see \cite{CMVk,FR,FRs,Lions,R}.

\vskip 2mm
We point out that the computation in \cite{BT} is formal.
The aim of this paper is to give a rigorous proof and extend the result to high dimension.
To present our main result, we distinguish two case: $p<0$ and $p>0$.
In the case of $-N<p<0$ and $q>p$, we consider the following variational problem:
\begin{equation}\label{eq 4}
  {\rm minimizing} \quad  E[\rho]=\frac{1}{2}\int_{\mathbb{R}^N} \int_{\mathbb{R}^N}K(x-y)\rho(x)\rho(y)dxdy
  +\int_{\mathbb{R}^N}F(x)\rho(x)dx
\end{equation}
over
\begin{equation}\label{eq 5}
 \mathcal{D}_{M,m}:= \left\{\rho\in L^1(\mathbb{R}^N) \cap L^\infty(\mathbb{R}^N):\rho\geq 0,\ \|\rho\|_{L^\infty}\leq M,\|\rho\|_{L^1}=m,\ \rho\in L^1(\mathbb{R}^N,Wdx)\right\},
\end{equation}
where $M>0$, $m>0$ and
\begin{equation*}
W(x) =\left\{ \arraycolsep=1pt
\begin{array}{lll}
 \max\  \{F(x),(1+|x|)^q\}, \quad
 \quad {\rm if}\   q>0,\\[2mm]
 \phantom{    }
 F(x),\quad
 \quad  \phantom{zzzz  zzzzz zzzzzz } {\rm if}\ q<0.
\end{array}
\right.\qquad
\end{equation*}

In the case of $0<p<q$, we consider another variational problem:
\begin{equation}\label{eq 6}
  {\rm minimizing} \quad  E[\mu]=\frac{1}{2}\int_{\mathbb{R}^N} \int_{\mathbb{R}^N}K(x-y)d\mu(x)d\mu(y)
  +\int_{\mathbb{R}^N}F(x)d\mu(x),
\end{equation}
over probability measures $\mu\in \mathcal{P}(\mathbb{R}^N):=\{\mu\geq0,\ \int_{\mathbb{R}^N}d\mu=1\}$,
endowed with the weak-$\ast$ topology.

Our results read as follows.
\begin{thm}\label{th1}
Assume that  $-N<p<0$, $q>p$, $K$ is a power-law repulsive-attractive potential verifying (\ref{eq d1})
and $\mathcal{D}_{M,m}$ is denoted  in (\ref{eq 5}) with $M,\, m>0$.
Then for any $M,\, m>0$, the nonlocal interaction functional $E[\cdot]$ (\ref{eq 4}) has
at least a minimizer in $\mathcal{D}_{M,m}$.
\end{thm}

\begin{thm}\label{th3}
Assume that $0<p<q$, $K$ is a power-law repulsive-attractive potential verifying (\ref{eq d1}).
Then  nonlocal interaction functional $E[\cdot]$ (\ref{eq 6}) has
at least a minimizer in probability measures $\mathcal{P}(\mathbb{R}^N)$.
\end{thm}

\begin{thm}\label{th4}
Let $q=2$, $p=2-N$ ($N>2$) and denote by $\omega_N$ the volume of unit ball in $\mathbb{R}^N$
and by $\chi$ the characteristic function of a domain. Assume that $F(x)=\beta|x|^2$ with $\beta>0$.
Then for any $m>0$ and $M\geq\frac{m+2\beta}{\omega_N}$, the function $\rho_0(x)=\frac{m+2\beta}{\omega_N}
\chi_{B(0,r_0)}(x)$ is the global minimizer of the nonlocal interaction functional $E[\cdot]$ (\ref{eq 4})
in $\mathcal{D}_{M,m}$, where $r_0=(\frac{m}{m+2\beta})^{1/N}$.
\end{thm}

Note that the energy functional with exogenous potential breaks the translation invariance,
so the traditional method to obtain the tightness is not useful for our model.
To this end, we establish the compactness of energy-minimizing sequences by showing
that the mass can not escape to infinity.

\vskip 2mm

The rest of our paper is organized as follows.
In Section 2, we introduce some preliminary properties which play an important role in obtaining the existence of minimizers.
Section 3 is devoted to proving Theorem \ref{th1}.
Finally,  the proof of Theorem \ref{th3} and Theorem \ref{th4} are given in Section 4 and  in Section 5 respectively.

\section{Preliminaries}
In this section, we first recall the following lemma,
which is used to prove weak lower-semicontinuity
of the functional $E(\rho)$ defined in
(\ref{eq 1}).

\begin{lem}\label{lemma 2}\cite[Lemma 3.3]{CFT}
Let $\{\rho_n\}_{n\in \mathbb{N}} \subset \mathcal{D}_{M,m}$ and $\rho\in \mathcal{D}_{M,m}$  such that  $\rho_n\rightharpoonup \rho$
in $L^s(\mathbb{R}^N)$ for some $s\in(1,+\infty)$. Then
\begin{eqnarray}\label{eq g4}
\lim_{n\rightarrow\infty}\int_{\mathbb{R}^N}\int_{\mathbb{R}^N}|x-y|^\gamma\rho_n(x)\rho_n(y)dxdy=
\int_{\mathbb{R}^N}\int_{\mathbb{R}^N}|x-y|^\gamma\rho(x)\rho(y)dxdy,
\end {eqnarray}
where $-N<\gamma<0$.
\end{lem}

The following is a special form of Hardy-Littlewood-Sobolev inequality, which is of vital importance to
obtain the lower bound of the functional $E[\rho]$.
\begin{prop}\label{proposition 2.3} \cite[Theorem 3.1]{Lieb}
For any $\gamma\in(-N,0)$ and $\rho\in L^{\frac{2N}{2N+\gamma}}(\mathbb{R}^N)$, we have
$$\int_{\mathbb{R}^N}\int_{\mathbb{R}^N}|x-y|^\gamma\rho(x)\rho(y)dxdy\leq
C(\gamma)\|\rho\|^2_{L^{\frac{2N}{2N+\gamma}}(\mathbb{R}^N)},$$
where the sharp constant $C(\gamma)$ is given by
$$C(\gamma)=\pi^{-\frac{\gamma}{2}}\frac{\Gamma(\frac{N}{2}+\frac{\gamma}{2})}
{\Gamma(N+\frac{\gamma}{2})}\left(\frac{\Gamma(\frac{N}{2})}{\Gamma(N)}\right)^{-1-\frac{\gamma}{N}}$$
with $\Gamma(\cdot)$ denoting the Gamma function.
\end{prop}
As a preparation for the proof of the existence of minimizers for (\ref{eq 4}), we need the
following version of concentration compactness principle.
\begin{lem}\label{lemma a4}
Assume that $-N<p<0<q$.
Let $\{\rho_n\}\subset \mathcal{D}_{M,m}$ be a minimizing sequence for (\ref{eq 4}).
Then there exists a subsequence $\{\rho_{n_k}\}$ satisfying:
there exists a bounded sequence $\{y_{n_k}\} \subset \mathbb{R}^N$
such that for all $\varepsilon>0$, there is $R>0$ with the property that
\begin{eqnarray}\label{eq abh1}
\int_{B(y_{n_k},R)}\rho_{n_k}(x)dx\geq m-\varepsilon\ \ \ for\ all\ k.
\end{eqnarray}
\end {lem}
\begin{proof}
We use the original idea of the proof of \cite[Lemma I.1]{Lions}.
For convenience, we denote
$$Q_n(R)=\sup_{y\in\mathbb{R}^N}\int_{B(y,R)}\rho_n(x)dx.$$
Noticing that $\{Q_n\}$ is a sequence of nondecreasing, nonnegative bounded
functions on $[0,+\infty)$ with $\lim _{R\rightarrow\infty}Q_n(R)=m$,
there exists a subsequence $\{Q_{n_k}\}$ and a nondecreasing nonnegative
function $Q$ such that $Q_{n_k}(R)\rightarrow Q(R)$ as $k\rightarrow\infty$, for $R>0$.
Let
\begin{eqnarray}\label{eq  ab2}
\alpha=\lim _{R\rightarrow\infty}Q(R).
\end {eqnarray}
Clearly $0\leq \alpha\leq m$.
In the following, we shall prove $\alpha=m$.
If $0\leq\alpha<m$,
then for $k$ and $R$ large enough, we have
$$\int_{B(0,R)}\rho_{n_k}(x)dx\leq\sup_{y\in\mathbb{R}^N}\int_{B(y,R)}\rho_{n_k}(x)dx\leq\frac{m+\alpha}{2}.$$
Since $\rho_{n_k}\in \mathcal{D}_{M,m}$, it follows that
$$m=\int_{\mathbb{R}^N}\rho_{n_k}(x)dx= \int_{B(0,R)}\rho_{n_k}(x)dx+\int_{\mathbb{R}^N\backslash B(0,R)}\rho_{n_k}(x)dx.$$
Thus
\begin{equation}\label{eq g1}
\int_{\mathbb{R}^N\backslash B(0,R)}\rho_{n_k}(x)dx\geq \frac{m-\alpha}{2}\ \ {\rm for}  \ k,R \ {\rm large \ enough}.
\end{equation}
Since $F(x)$ is a continuous function and satisfies $F(x)\rightarrow+\infty$ as $|x|\rightarrow+\infty$,
for sufficiently large $R$, we get $F(x)\geq\frac{2c_0+2}{m-\alpha}>0$ for all
$x\in \mathbb{R}^N\backslash B(0,R)$, where
\begin{equation}\label{eq a5}
 c_0=\inf \{E(\rho):\rho\in \mathcal{D}_{M,m}\}\geq0.
\end{equation}
Moreover, (\ref{eq g1}) yields that
\begin{eqnarray*}
E[\rho_{n_k}]&\geq &\int_{\mathbb{R}^N}F(x)\rho_{n_k}(x)dx
\geq \int_{\mathbb{R}^N\backslash B(0,R)}F(x)\rho_{n_k}(x)dx
\\&\geq &\frac{2c_0+2}{m-\alpha}\int_{\mathbb{R}^N\backslash B(0,R)}\rho_{n_k}(x)dx
\geq \frac{2c_0+2}{m-\alpha}\cdot \frac{m-\alpha}{2}\geq c_0+1,
\end{eqnarray*}
which contradicts with the fact that $\{\rho_{n_k}\}$ is a minimizing sequence.
Therefore, we infer  $\alpha= m$.

We next prove (\ref{eq abh1}).
Since $\lim _{R\rightarrow\infty}Q(R)=m$, for some $R_0>0$ we have $Q(R_0)>\frac{m}{2}$. For any $k\in \mathbb{N}$, let
$\{y_{n_k}\} \subset \mathbb{R}^N$ satisfy
\begin{eqnarray}\label{eq abc3}
Q_{n_k}(R_0)\leq\int_{B(y_{n_k},R_0)}\rho_{n_k}(x)dx+\frac{1}{k}.
\end{eqnarray}
Now for $0<\varepsilon<\frac{m}{2}$, fix $R$ such that $Q(R)>m-\varepsilon>\frac{m}{2}$
and choose $\{x_{n_k}\}\subset \mathbb{R}^N$ to satisfy
\begin{eqnarray}\label{eq abc4}
Q_{n_k}(R)\leq\int_{B(x_{n_k},R)}\rho_{n_k}(x)dx+\frac{1}{k}.
\end{eqnarray}
By virtue of (\ref{eq abc3}) and (\ref{eq abc4}) we see that for $k$
large enough
$$\int_{B(y_{n_k},R_0)}\rho_{n_k}(x)dx+\int_{B(x_n,R)}\rho_{n_k}(x)dx\geq Q(R_0)+Q(R)>m=\int_{\mathbb{R}^N}\rho_{n_k}(x)dx.$$
It follows that for such $k$
$$B(x_{n_k},R)\cap B(y_{n_k},R_0)\neq\varnothing.$$
Because of $B(x_{n_k},R)\subset B(y_{n_k},2R+R_0)$, we deduce
$$\int_{B(y_{n_k},2R+R_0)}\rho_{n_k}(x)dx\geq\int_{B(y_{n_k},R_0)}\rho_{n_k}(x)dx+\int_{B(x_n,R)}\rho_{n_k}(x)dx-
\int_{B(x_{n_k},R)\cap B(y_{n_k},R_0)}\rho_{n_k}(x)dx.$$
That is
$$m-\varepsilon\leq\int_{B(y_{n_k},2R+R_0)}\rho_{n_k}(x)dx$$
for sufficiently large $k$. Choosing $R$ even larger, if necessary,
we can achieve that (\ref{eq abh1}) holds for all $k$.

Finally we claim that $\{y_{n_k}\}$ is bounded. Otherwise, assuming that
$\{y_{n_k}\}$ is unbounded,
we may take a subsequence of $\{y_{n_k}\}$, still denoted by
$\{y_{n_k}\}$, such that
$|y_{n_k}|\rightarrow+\infty$  ${\rm as}\ k\rightarrow+\infty$.
Then for large enough $k$, since $|x|\rightarrow+\infty$ for all $x\in B(y_{n_k},R)$,
so $F(x)\geq\frac{c_0+1}{m-\varepsilon}>0$ ($\varepsilon<m$) for all $x\in B(y_{n_k},R)$,
where $c_0$ is given by (\ref{eq a5}).
Therefore, this implies that
\begin{eqnarray*}
E[\rho_{n_k}]&\geq &\int_{\mathbb{R}^N}F(x)\rho_{n_k}(x)dx\geq \int_{B(y_k,R)}F(x)\rho_{n_k}(x)dx
\geq \frac{c_0+1}{m-\varepsilon}\int_{B(y_k,R)}\rho_{n_k}(x)dx
\\&\geq &\frac{c_0+1}{m-\varepsilon}\cdot (m-\varepsilon)\geq c_0+1,
\end{eqnarray*}
which contradicts with the fact that $\{\rho_{n_k}\}$ is the minimizing sequence. Consequently,
$\{y_{n_k}\}$ is bounded,
whereby the proof is complete.
\end {proof}
Next, we can use Lemma \ref{lemma a4} to establish the compactness of energy-minimizing sequences.
\begin{lem}\label{lemma a3}
Let $\{\rho_n\}\subset \mathcal{D}_{M,m}$
and assume that there exists a bounded sequence $\{y_n\} \subset \mathbb{R}^N$
such that for all $\varepsilon>0$, there is $R>0$ satisfying
\begin{eqnarray*}
\int_{B(y_n,R)}\rho_n(x)dx\geq m-\varepsilon\ \ \ for\ all\ n.
\end{eqnarray*}
Then there exists a subsequence $\{\rho_{n_k}\} \subset \mathcal{D}_{M,m}$ and
$\rho_0\in \mathcal{D}_{M,m}$, such that
\begin{eqnarray*}
\rho_{n_k}\rightharpoonup \rho_0\ \ in  \ \  L^s(\mathbb{R}^N), \ \ \ as\ k\rightarrow \infty
\end {eqnarray*}
for some $s\in(1,+\infty)$.
\end{lem}
\begin{proof}
Since $\{\rho_n\}\subset \mathcal{D}_{M,m}$, then all members of the sequence are uniformly
bounded in $L^1(\mathbb{R}^N) \cap L^\infty(\mathbb{R}^N)$, so $\{\rho_n\}$
is uniformly bounded in $L^s(\mathbb{R}^N)$ by the interpolation inequality, for all $s\in(1,+\infty)$.
Hence there exists $\rho_0\in L^s(\mathbb{R}^N)$ such that
\begin{equation}\label{eq f1}
\rho_{n_k}\rightharpoonup \rho_0\ {\rm\ in} \ \  L^s(\mathbb{R}^N), {\rm\ \ as}\ k\rightarrow \infty.
\end{equation}

We use five steps to prove  $\rho_0\in \mathcal{D}_{M,m}$.

\emph{Step 1.} Proof of $\rho_0\geq 0$ a.e. in $\mathbb{R}^N$. Let $S$ be any bounded set in $\mathbb{R}^N$,
and $\chi_S$ be the characteristic function. Since $\chi_S\in L^\frac{s}{s-1}(\mathbb{R}^N)$, by (\ref{eq f1})
we derive
\begin{eqnarray*}
\int_{\mathbb{R}^N}\rho_{n_k}(x)\chi_S(x)dx\rightarrow \int_{\mathbb{R}^N}\rho_0(x)\chi_S(x)dx,
{\rm\ \ as}\ k\rightarrow \infty,
\end{eqnarray*}
which implies that
\begin{eqnarray}\label{eq i1}
\int_S\rho_{n_k}(x)dx\rightarrow \int_S\rho_0(x)dx, {\rm\ \ as}\ k\rightarrow \infty.
\end{eqnarray}
Let $S=\{x\in \mathbb{R}^N; \rho_0(x)<0\}\cap B(0,L)$, where $L>0$, then on the one hand, we have
$$\int_S\rho_0(x)dx\leq0.$$
On the other hand,
$$\int_S\rho_0(x)dx=\lim_{k\rightarrow\infty}\int_S\rho_{n_k}(x)dx\geq0.$$
Therefore $\int_S\rho_0(x)dx=0$,
which implies that the measure of $S$ is zero.
Hence $\rho_0\geq 0$ a.e.  in $\mathbb{R}^N$.

\emph{Step 2.} Proof of $\rho_0\in L^1(\mathbb{R}^N)$.
By (\ref{eq i1}) with $S$ replaced by $B_L(0)$, we have
\begin{eqnarray*}
\int_{B(0,L)}\rho_{n_k}(x)dx\rightarrow \int_{B(0,L)}\rho_0(x)dx, {\rm\ \ as}\ k\rightarrow \infty.
\end{eqnarray*}
Thus, for each $\varepsilon>0$, there exists $k_0>0$, such that when $k>k_0$, we have
$$\int_{B(0,L)}\rho_0(x)dx\leq\int_{B(0,L)}\rho_{n_k}(x)dx+\varepsilon\leq\int_{\mathbb{R}^N}\rho_{n_k}(x)dx
+\varepsilon=m+\varepsilon.$$
Passing to the limit as $L\rightarrow+\infty$,
$$\int_{\mathbb{R}^N}\rho_0(x)dx\leq\liminf_{L\rightarrow\infty}\int_{B(0,L)}\rho_0(x)dx\leq m+\varepsilon,$$
which implies that $\rho_0\in L^1(\mathbb{R}^N)$ and $\int_{\mathbb{R}^N}\rho_0(x)dx\leq m$ by the arbitrariness
of $\varepsilon$.

\emph{Step 3.} To prove $\int_{\mathbb{R}^N}\rho_0(x)dx=m$. By the assumption of the lemma,
there exists $\{y_{n_k}\} \subset \mathbb{R}^N$ such that
\begin{eqnarray*}
{\rm {\ for\ all}}\ \varepsilon>0,\ \  m\geq\int_{B(y_{n_k},R)} \rho_{n_k}(x)dx\geq m-\varepsilon
\ \ {\rm {for\ some}}\ R>0 .
\end{eqnarray*}
Denote by $A$ the bound of $\{y_{n_k}\}$.
Putting $\bar{R}=R+A$, we have $ B(y_{n_k},R)\subset B(0,\bar{R})$, then
\begin{eqnarray}\label{eq k1}
m\geq\int_{B(0,\bar{R})}\rho_{n_k}(x)dx\geq m-\varepsilon.
\end{eqnarray}
Since $\{\rho_{n_k}\}\subset\mathcal{D}_{M,m}$, we can infer from (\ref{eq k1}) that
\begin{eqnarray}\label{eq k2}
\int_{\mathbb{R}^N\backslash B(0,\bar{R})}\rho_{n_k}(x)dx=\int_{\mathbb{R}^N}\rho_{n_k}(x)dx-\int_ {B(0,\bar{R})}
\rho_{n_k}(x)dx\leq m-(m-\varepsilon)=\varepsilon.
\end{eqnarray}
By (\ref{eq i1}) with $S$ replaced by $B_{\bar{R}}(0)$, we have
\begin{eqnarray}\label{eq k4}
\int_{B(0,\bar{R})}\rho_{n_k}(x)dx\rightarrow\int_{B(0,\bar{R})}\rho_0(x)dx,\ \ {\rm as} \ k\rightarrow\infty.
\end{eqnarray}
By virtue of (\ref{eq k1}) we obtain
\begin{eqnarray}\label{eq k5}
m\geq\int_{B(0,\bar{R})}\rho_0(x)dx\geq m-\varepsilon,
\end{eqnarray}
which implies
\begin{eqnarray}\label{eq k7}
\int_{\mathbb{R}^N\backslash B(0,\bar{R})}\rho_0(x)dx\leq \varepsilon.
\end{eqnarray}
Hence, we can infer from (\ref{eq k2}) (\ref{eq k4}) (\ref{eq k7}) that for $k$ large enough,
\begin{eqnarray*}
\left|\int_{\mathbb{R}^N}(\rho_{n_k}(x)-\rho_0(x))dx\right|&\leq&\left|\int_{B(0,\bar{R})}(\rho_{n_k}(x)-\rho_0(x))dx\right|
+\left|\int_{\mathbb{R}^N\backslash B(0,\bar{R})}(\rho_{n_k}(x)-\rho_0(x))dx\right|\\
\\&\leq &\varepsilon+\int_{\mathbb{R}^N\backslash B(0,\bar{R})}\rho_{n_k}(x)dx+\int_{\mathbb{R}^N\backslash
B(0,\bar{R})}\rho_0(x)dx\leq 3\varepsilon.
\end{eqnarray*}
Then we have
$$\int_{\mathbb{R}^N}\rho_0(x)dx=\lim_{k\rightarrow\infty}\int_{\mathbb{R}^N}\rho_{n_k}(x)dx=m.$$

\emph{Step 4.} To prove that $\rho_0\in L^\infty(\mathbb{R}^N)$ and $\|\rho_0\|_{L^\infty}\leq M$.
Let
$$S=\{x\in \mathbb{R}^N:  \rho_0(x)>M\}\cap B(0,L),$$
then $\int_S(M-\rho_0(x))dx\leq0$.
On the other hand,
$$\int_S(M-\rho_0(x))dx=\lim_{k\rightarrow\infty}\int_S(M-\rho_{n_k})(x)dx\geq0.$$
Therefore
$$\int_S(M-\rho_0(x))dx=0,$$
which implies that the measure of $S$ is zero.
Hence $\rho_0\in L^\infty(\mathbb{R}^N)$ and $\|\rho_0\|_{L^\infty}\leq M$.

\emph{Step 5.} To prove $\rho_0\in L^1(\mathbb{R}^N,Wdx)$.
Similarly to (\ref{eq i1}), we have
$$\int_{B(0,L)}\rho_{n_k}(x)W(x)dx\rightarrow\int_{B(0,L)}\rho_0(x)W(x)dx,\ \ {\rm as} \ k\rightarrow\infty.$$
Then there exists $k_0$ such that
$$\int_{B(0,L)}\rho_0(x)W(x)dx\leq\int_{B(0,L)}\rho_{n_{k_0}}(x)W(x)dx+1\leq\int_{\mathbb{R}^N}
\rho_{n_{k_0}}(x)W(x)dx+1.$$
Letting $L\rightarrow+\infty$, we obtain
$$\int_{\mathbb{R}^N}\rho_0(x)W(x)dx\leq
\int_{\mathbb{R}^N}\rho_{n_{k_0}}(x)W(x)dx
+1,$$
therefore, $\rho_0\in L^1(\mathbb{R}^N,Wdx)$.

In conjunction with \emph{Steps 1-5}, it follows that $\rho_0\in \mathcal{D}_{M,m}$.
\end{proof}

\section{Proof of Theorem \ref{th1}}
In this section, we prove Theorem \ref{th1}.
We divide the proof into two parts according to the parameter regime of $p$ and $q$.
We first consider the case: $-N<p<0<q$.
The following lemma derives weak lower-semicontinuity of the attractive part of
the functional $E[\rho]$.
\begin{lem}\label{lemma a5}
Let $\{\rho_n\} \subset \mathcal{D}_{M,m}$ and $\rho_0\in \mathcal{D}_{M,m}$  such that
$\rho_n\rightharpoonup \rho_0$
in $L^s(\mathbb{R}^N)$ for some $s\in(1,+\infty)$. Then
\begin{equation}\label{eq 8}
\int_{\mathbb{R}^N} \int_{\mathbb{R}^N}|x-y|^q\rho_0(x)\rho_0(y)dxdy
\leq\liminf_{n\rightarrow\infty}\int_{\mathbb{R}^N} \int_{\mathbb{R}^N}|x-y|^q\rho_n(x)\rho_n(y)dxdy,
\end {equation}
where $q>0$.
\end {lem}

\begin{proof}
For any fixed $R>0$, we consider
$$H_n(x)=\int_{B(0,R)}|x-y|^q\rho_n(y)dy {\rm\ \ and}\ \  H_0(x)=\int_{B(0,R)}|x-y|^q\rho_0(y)dy.$$
Since $q>0$, for $x\in B(0,R)$ we have
$$\int_{B(0,R)}|x-y|^q\rho_0(y)dy\leq(2R)^qm,$$
which implies that $H_0\in L^\infty(B(0,R))$ and $H_0\in L^\frac{s}{s-1}(B(0,R))$.
By $\rho_n\rightharpoonup \rho_0$ in $L^s(B(0,R))$, we deduce that
\begin{eqnarray}\label{eq2-lemma a5}
\int_{B(0,R)}H_0(x)\left(\rho_n(x)-\rho_0(x)\right)\rightarrow0,\ \ {\rm as} \ n\rightarrow\infty.
\end{eqnarray}
Noticing that $\rho_n\in L^\infty(\mathbb{R}^N)$ and $\int_{B(0,R)}|\cdot-y|^qdy\in L^\frac{s}{s-1}(B(0,R))$,
it follows that
\begin{eqnarray}
&&\int_{B(0,R)}\rho_n(x)\left(H_n(x)-H_0(x)\right)dx\nonumber
\\&&=\int_{B(0,R)}\rho_n(x)\left(\int_{B(0,R)}|x-y|^q[\rho_n(y)-\rho_0(y)]dy\right)dx\nonumber
\\&&\leq\|\rho_n\|_{L^\infty(\mathbb{R}^N)}\int_{B(0,R)}\left(\int_{B(0,R)}|x-y|^qdy\right)[\rho_n(x)-\rho_0(x)]dx
\rightarrow0,\label{eq3-lemma a5}
\end{eqnarray}
as $n\rightarrow\infty$. Upon an application of (\ref{eq2-lemma a5}) together with(\ref{eq3-lemma a5}) we see that
\begin{eqnarray*}
&&\int_{B(0,R)}H_n(x)\rho_n(x)dx
\\&&=\int_{B(0,R)}H_0(x)\left(\rho_n(x)-\rho_0(x)\right)dx
+\int_{B(0,R)}\rho_n(x)\left(H_n(x)-H_0(x)\right)dx
\\&&+\int_{B(0,R)}H_0(x)\rho_0(x)dx
\\&&\rightarrow\int_{B(0,R)}H_0(x)\rho_0(x)dx,\ \ {\rm as} \ n\rightarrow\infty,
\end{eqnarray*}
i.e.
\begin{equation}\label{eq f2}
\lim_{n\rightarrow\infty}\int_{B(0,R)}\int_{B(0,R)}|x-y|^q\rho_n(x)\rho_n(y)dxdy=\int_{B(0,R)}
\int_{B(0,R)}|x-y|^q\rho_0(x)\rho_0(y)dxdy.
\end{equation}

Since $\rho_0\in L^1(\mathbb{R}^N,Wdx)$, we see that for any $\varepsilon>0$ small,
there exist $R$ large enough, such that
\begin{eqnarray*}%\label{eq h1}
\int_{\mathbb{R}^N\backslash B(0,R)}(1+|x|)^q\rho_0(x)dx<\varepsilon.
\end{eqnarray*}
This entails that
\begin{eqnarray*}
&&\int_{B(0,R)}\int_{\mathbb{R}^N\backslash B(0,R)}|x-y|^q\rho_0(x)\rho_0(y)dxdy
\\&&\leq\int_{B(0,R)}\left(\int_{\mathbb{R}^N\backslash B(0,R)}(1+|x|)^q\rho_0(x)dx\right)(1+|y|)^q\rho_0(y)dy
\leq (1+R)^q m\varepsilon,
\end{eqnarray*}
and
\begin{eqnarray*}
&&\int_{\mathbb{R}^N\backslash B(0,R)}\int_{\mathbb{R}^N\backslash B(0,R)}|x-y|^q\rho_0(x)\rho_0(y)dxdy
\\&&\leq\int_{\mathbb{R}^N\backslash B(0,R)}\left(\int_{\mathbb{R}^N\backslash B(0,R)}(1+|x|)^q\rho_0(x)dx\right)
(1+|y|)^q\rho_0(y)dy
\leq \varepsilon^2
\end{eqnarray*}
because of $E(\rho_0)<+\infty$.
In conjunction with the above two inequalities and (\ref{eq f2}) we obtain
\begin{eqnarray*}
&&\int_{\mathbb{R}^N} \int_{\mathbb{R}^N}|x-y|^q\rho_0(x)\rho_0(y)dxdy
\\&&\leq\int_{B(0,R)}\int_{B(0,R)}|x-y|^q\rho_0(x)\rho_0(y)dxdy+2(1+R)^qm\varepsilon+\varepsilon^2
\\&&\leq\liminf_{n\rightarrow\infty}\int_{B(0,R)}\int_{B(0,R)}|x-y|^q\rho_n(x)\rho_n(y)dxdy+2(1+R)^qm\varepsilon
+\varepsilon^2
\\&&\leq\liminf_{n\rightarrow\infty}\int_{\mathbb{R}^N}\int_{\mathbb{R}^N}|x-y|^q\rho_n(x)\rho_n(y)dxdy+2(1+R)^q
m\varepsilon+\varepsilon^2.
\end{eqnarray*}
Taking $\varepsilon\rightarrow0$, we obtain (\ref{eq 8}), which completes the proof.
\end{proof}
We can prove Theorem \ref{th1}  in the case of  $-N<p<0<q$.
\begin{prop}\label{prop1}
Assume that  $-N<p<0<q$.
Then for any $M,\, m>0$, the nonlocal interaction functional $E[\cdot]$  has
at least a minimizer in $\mathcal{D}_{M,m}$.
\end {prop}

\begin{proof}
Let $\{\rho_n\}\subset \mathcal{D}_{M,m}$ be a minimizing sequence for (\ref{eq 4}).
By Lemma \ref{lemma a4} and Lemma \ref{lemma a3}, there exists a
subsequence $\{\rho_{n_k}\}\subset \mathcal{D}_{M,m}$
and $\rho_0\in \mathcal{D}_{M,m}$ such that
\begin{eqnarray*}
\rho_{n_k}\rightharpoonup \rho_0 {\rm\ \ in}\ \  L^s(\mathbb{R}^N), \ {\rm\ \ as}\ k\rightarrow \infty
\end {eqnarray*}
for some $s\in(1,+\infty)$.
Clearly we have that
\begin{eqnarray}\label{eq ag1}
E[\rho_0]\geq\inf \{E[\rho]:\rho\in \mathcal{D}_{M,m}\}=\lim_{k\rightarrow\infty}E[\rho_{n_k}].
\end{eqnarray}
In order to prove that $\rho_0$ is a minimizer in $\mathcal{D}_{M,m}$,
we must show that the functional $E[\rho]$ is weak lower semi-continuous.
Similarly to (\ref{eq i1}), we have
$$\int_{B(0,L)}\rho_{n_k}(x)F(x)dx\rightarrow\int_{B(0,L)}\rho_0(x)F(x)dx,\ \ {\rm as} \ k\rightarrow\infty.$$
Thus, for any $\varepsilon>0$, there exists $k_0>0$, such that for $k>k_0$, we have
$$\int_{B(0,L)}F(x)\rho_0(x)dx\leq\int_{B(0,L)}\rho_{n_k}(x)F(x)dx+\varepsilon\leq\int_{\mathbb{R}^N}
\rho_{n_k}(x)F(x)dx+\varepsilon.$$
Since $\int_{B(0,L)}F(x)\rho_0(x)dx$ is an increasing function of $L$, it follows that
$$\int_{\mathbb{R}^N}F(x)\rho_0(x)dx\leq\liminf_{L\rightarrow\infty}\int_{B(0,L)}F(x)\rho_0(x)dx\leq
\int_{\mathbb{R}^N}F(x)\rho_{n_k}(x)dx+\varepsilon.$$
Passing to the limit as $k\rightarrow+\infty$, we obtain
\begin{eqnarray}\label{eq 9}
\int_{\mathbb{R}^N}F(x)\rho_0(x)dx\leq\liminf_{k\rightarrow\infty}\int_{\mathbb{R}^N}F(x)\rho_{n_k}(x)dx+\varepsilon.
\end {eqnarray}
Invoking (\ref{eq g4}) along with (\ref{eq 8}), (\ref{eq 9}) and letting $\varepsilon\rightarrow0$ we deduce that
\begin{eqnarray}\label{eq g3}
E[\rho_0]\leq\liminf_{k\rightarrow\infty}E[\rho_{n_k}].
\end {eqnarray}
From (\ref{eq ag1}) and (\ref{eq g3}) we can get
\begin{eqnarray*}\label{eq g6}
E[\rho_0]=\inf \{E[\rho]:\rho\in \mathcal{D}_{M,m}\},
\end{eqnarray*}
which entails that $\rho_0$ is a minimizer for  (\ref{eq 4}) in $\mathcal{D}_{M,m}$ when $-N<p<0<q$.
\end {proof}

\vskip 3mm
Next considering the case: $-N<p<q<0$,  we  need the
following version of concentration compactness principle.
\begin{lem}\label{lemma b1}
Assume that $-N<p<q<0$.
Let $\{\rho_n\}\subset \mathcal{D}_{M,m}$ be a minimizing sequence for (\ref{eq 4}).
Then there exists a subsequence $\{\rho_{n_k}\}$ satisfying:
there exists a bounded sequence $\{y_{n_k}\} \subset \mathbb{R}^N$
such that for all $\varepsilon>0$, there is $R>0$ with the property that
\begin{eqnarray}\label{eq abh1}
\int_{B(y_{n_k},R)}\rho_{n_k}(x)dx\geq m-\varepsilon\ \ \ for\ all\ k.
\end{eqnarray}
\end {lem}
\begin{proof}
The proof is similar to that of  Lemma \ref{lemma a4}.
Noting that $0\leq \alpha\leq m$, where $\alpha$ is given in (\ref{eq ab2}).
In the following, we shall prove $\alpha=m$.
If $0\leq\alpha<m$, (\ref{eq g1}) holds.
%recalling (\ref{eq g1})
In view of Proposition \ref{proposition 2.3}, for $\rho\in \mathcal{D}_{M,m}$ we have
\begin{eqnarray*}
&&\int_{\mathbb{R}^N} \int_{\mathbb{R}^N}|x-y|^q\rho(x)\rho(y)dxdy
\\&&\leq C(q)\|\rho\|^2_{L^{\frac{2N}{2N+q}}(\mathbb{R}^N)}
\leq C(q)\left(\|\rho\|_{L^\infty(\mathbb{R}^N)}^{\frac{2N}{2N+q}-1}\int_{\mathbb{R}^N}\rho(x)dx\right)^{2\cdot\frac{2N+q}{2N}}
\leq  C(q)M^\frac{-q}{N}m^\frac{2N+q}{N}.
\end{eqnarray*}
Since $F(x)$ is a continuous function and satisfies $F(x)\rightarrow+\infty$ as $|x|\rightarrow+\infty$,
for sufficiently large $R$, we have $F(x)\geq\frac{2c_0+2-\frac{1}{q}C(q)M^\frac{-q}{N}m^\frac{2N+q}{N}}
{m-\alpha}>0$ for all $x\in \mathbb{R}^N\backslash B(0,R)$, where
\begin{equation*}
 c_0=\inf \{E(\rho):\rho\in \mathcal{D}_{M,m}\}\geq\frac{1}{2q}C(q)M^\frac{-q}{N}m^\frac{2N+q}{N}.
\end{equation*}
Moreover, from (\ref{eq g1}) we see that
\begin{eqnarray*}
E[\rho_{n_k}]&\geq &\frac{1}{2q}\int_{\mathbb{R}^N} \int_{\mathbb{R}^N}|x-y|^q\rho_{n_k}(x)\rho_{n_k}(y)dxdy+\int_{\mathbb{R}^N}F(x)\rho_{n_k}(x)dx\nonumber
\\&\geq &\frac{1}{2q}C(q)M^\frac{-q}{N}m^\frac{2N+q}{N}+\frac{2c_0+2-\frac{1}{q}C(q)M^\frac{-q}{N}m^\frac{2N+q}{N}}
{m-\alpha}\int_{\mathbb{R}^N\backslash B(0,R)}\rho_{n_k}(x)dx\nonumber
\geq c_0+1,
\end{eqnarray*}
which contradicts with the fact that $\{\rho_{n_k}\}$ is a minimizing sequence.
Consequently, $\alpha= m$.
Thus, (\ref{eq abh1}) holds %The proof is same to Lemma \ref{lemma a4}, here we omit it.
and  we claim that $\{y_{n_k}\}$ is bounded. Otherwise, assuming that
$\{y_{n_k}\}$ is unbounded,
we may take a subsequence of $\{y_{n_k}\}$, still denoted by
$\{y_{n_k}\}$, such that
$|y_{n_k}|\rightarrow+\infty$  ${\rm as}\ k\rightarrow+\infty$.
Then for large enough $k$, since $|x|\rightarrow+\infty$ for all $x\in B(y_{n_k},R)$,
so $F(x)\geq\frac{c_0+1-\frac{1}{2q}C(q)M^\frac{-q}{N}m^\frac{2N+q}{N}}
{m-\varepsilon}>0$ ($\varepsilon<m$) for all $x\in B(y_{n_k},R)$.
Therefore, this implies that
\begin{eqnarray*}
E[\rho_{n_k}]&\geq &\frac{1}{2q}\int_{\mathbb{R}^N} \int_{\mathbb{R}^N}|x-y|^q\rho_{n_k}(x)\rho_{n_k}(y)dxdy
+ \int_{B(y_k,R)}F(x)\rho_{n_k}(x)dx\nonumber
\\&\geq &\frac{1}{2q}C(q)M^\frac{-q}{N}m^\frac{2N+q}{N}+\frac{c_0+1-\frac{1}{2q}C(q)M^\frac{-q}{N}m^\frac{2N+q}{N}}
{m-\varepsilon}\int_{B(y_k,R)}\rho_{n_k}(x)dx
\\&\geq & c_0+1
\end{eqnarray*}
because of (\ref{eq abh1}),
which contradicts with the fact that $\{\rho_{n_k}\}$ is the minimizing sequence. Thus,
we achieve that $\{y_{n_k}\}$ is bounded. This finishes the proof of the lemma.
\end {proof}

Obviously, a direct application of Lemma \ref{lemma b1}
enables us to get  Lemma \ref{lemma a3}
and  establish the compactness of energy-minimizing sequences.

We can prove Theorem \ref{th1}  in the case of $-N<p<q<0$.
\begin{prop}\label{prop2}
Assume that  $-N<p<q<0$.
Then for any $M,\, m>0$, the nonlocal interaction functional $E[\cdot]$  has
at least a minimizer in $\mathcal{D}_{M,m}$.
\end {prop}
\begin{proof}
Let $\{\rho_n\}\subset \mathcal{D}_{M,m}$ be a minimizing sequence for (\ref{eq 4}).
By Lemma \ref{lemma b1} and Lemma \ref{lemma a3}, there exists a
subsequence $\{\rho_{n_k}\}\subset \mathcal{D}_{M,m}$
and $\rho_0\in \mathcal{D}_{M,m}$ such that
\begin{eqnarray*}
\rho_{n_k}\rightharpoonup \rho_0 {\rm\ \ in}\ \  L^s(\mathbb{R}^N), \ {\rm\ \ as}\ k\rightarrow \infty
\end{eqnarray*}
for some $s\in(1,+\infty)$.
Clearly we have that
\begin{eqnarray}\label{eq ab2g1}
E[\rho_0]\geq\inf \{E[\rho]:\rho\in \mathcal{D}_{M,m}\}=\lim_{k\rightarrow\infty}E[\rho_{n_k}].
\end{eqnarray}
In order to prove that $\rho_0$ is a minimizer in $\mathcal{D}_{M,m}$,
we must show that the functional $E[\rho]$ is weak lower semi-continuous.
From (\ref{eq g4}) in Lemma \ref{lemma 2}, we obtain
\begin{eqnarray}
&&\lim_{n\rightarrow\infty}\int_{\mathbb{R}^N}\int_{\mathbb{R}^N}(\frac{1}{q}|x-y|^q-\frac{1}{p}|x-y|^p)
\rho_n(x)\rho_n(y)dxdy\nonumber
\\&&=\int_{\mathbb{R}^N}\int_{\mathbb{R}^N}(\frac{1}{q}|x-y|^q-\frac{1}{p}|x-y|^p)\rho(x)\rho(y)dxdy.\label{eq abg4}
\end{eqnarray}
Combining (\ref{eq abg4}) and (\ref{eq 9}), and letting $\varepsilon\rightarrow0$ we know that
\begin{eqnarray}\label{eq hg3}
E[\rho_0]\leq\liminf_{k\rightarrow\infty}E[\rho_{n_k}].
\end{eqnarray}
In view of (\ref{eq ab2g1}) and (\ref{eq hg3}) we conclude that
\begin{eqnarray*}\label{eq hg6}
E[\rho_0]=\inf \{E[\rho]:\rho\in \mathcal{D}_{M,m}\},
\end{eqnarray*}
which means that $\rho_0$ is a minimizer for  (\ref{eq 4}) in $\mathcal{D}_{M,m}$ when $-N<p<q<0$.
\end {proof}

Now we can prove Theorem \ref{th1}.

\emph{Proof of Theorem \ref{th1}.}
With Proposition \ref{prop1} and Proposition \ref{prop2} at hand,
we can easily assert there  exists at least a minimizer for  (\ref{eq 4})
in $\mathcal{D}_{M,m}$ when $-N<p<0$ and $q>p$.\hfill$\Box$
\section{Proof of Theorem \ref{th3}}
In this section, we prove Theorem \ref{th3}.
Before going into details, let us first give the following lemma, which will play a key role in
the derivation of compactness of
the energy-minimizing sequence in probability measures.

\begin{lem}\label{lemma 4}
Assume that $0<p<q$.
Let $\{\mu_n\}\subset\mathcal{P}(\mathbb{R}^N)$
be a minimizing sequence for (\ref{eq 6}).
Then there is a subsequence $\{\mu_{n_k}\}$ satisfying:
there exists a bounded sequence $\{y_{n_k}\} \subset \mathbb{R}^N$
such that for all $\varepsilon>0$, there is $R>0$ with the property that
\begin{eqnarray}\label{eq abhi1}
\int_{B(y_{n_k},R)}\mu_{n_k}(x)dx\geq 1-\varepsilon\ \ \ for\ all\ k.
\end{eqnarray}
\end{lem}
\begin{proof}
The proof is based on \cite[Section 4.3]{Struwe}.
Denote
$$Q_{n_k}(R)=\sup_{y\in\mathbb{R}^N}\int_{B(y,R)}d\mu_{n_k}(x).$$
Noticing that $\{Q_n\}$ is a sequence of nondecreasing, nonnegative bounded
functions on $[0,+\infty)$ with $\lim _{R\rightarrow\infty}Q_n(R)=1$,
there exists a subsequence $\{Q_{n_k}\}$ and a nondecreasing nonnegative
function $Q$ such that $Q_{n_k}(R)\rightarrow Q(R)$ as $k\rightarrow\infty$, for $R>0$.
Let
$$\alpha=\lim _{R\rightarrow\infty}Q(R).$$
Clearly $0\leq \alpha\leq 1$.
Similarly to (\ref{eq g1}), if $\alpha<1$ we have
\begin{equation}\label{eq abh4}
\int_{\mathbb{R}^N\backslash B(0,R)}d\mu_{n_k}(x)\geq \frac{1-\alpha}{2}\ \ {\rm for}  \ k,R \ {\rm large \ enough}.
\end{equation}
Noting that $q>p>0$, we deduce
$K(x)\geq\frac{1}{q}-\frac{1}{p}$.
For sufficiently large $R$, we have $F(x)\geq\frac{2c_0+2-\frac{1}{q}+\frac{1}{p}}{1-\alpha}>0$ for all
$x\in \mathbb{R}^N\backslash B(0,R)$, where
\begin{equation*}%\label{eq abh3}
 c_0=\inf \{E[\mu]:\mu\in \mathcal{P}(\mathbb{R}^N)\}\geq\frac{1}{q}-\frac{1}{p}.
\end{equation*}
Using (\ref{eq abh4}) we obtain
\begin{eqnarray*}
E[\mu_{n_k}]&\geq &\frac{1}{2}(\frac{1}{q}-\frac{1}{p})+\int_{\mathbb{R}^N\backslash B(0,R)}F(x)d\mu_{n_k}(x)\nonumber
\\&\geq &\frac{1}{2}(\frac{1}{q}-\frac{1}{p})+\frac{2c_0+2-\frac{1}{q}+\frac{1}{p}}{1-\alpha}
\cdot\frac{1-\alpha}{2}\geq c_0+1,
\end{eqnarray*}
which contradicts with the fact that $\{\mu_{n_k}\}$ is a minimizing sequence.
Thus, $\alpha=1$. Similarly to the proof of Lemma \ref{lemma a4},
it follows that (\ref{eq abhi1}) holds.
Using the fact that $F(x)\geq\frac{c_0+1-\frac{1}{2}(\frac{1}{q}-\frac{1}{p})}
{1-\varepsilon}>0$ ($\varepsilon<1$)
for all $x\in B(y_{n_k},R)$ and (\ref{eq abhi1}), we obtain that $\{y_{n_k}\}$ is bounded.
Therefore the proof is complete.
\end{proof}
The following lemma plays a key role to derive weak lower-semicontinuity of
the functional $E(\mu)$.
\begin{lem}\label{lemma 5} \cite[Lemma 2.2]{SST}
Assume that the function $K:[0,\infty]\rightarrow (-\infty,\infty]$ is  lower-semicontinuous
and bounded from below. Then the energy $E:\mathcal{P}(\mathbb{R}^N)\rightarrow (-\infty,\infty]$ defined in
(\ref{eq 6}) with $F\equiv0$ is weakly lower-semicontinuous with respect to weak convergence of measures.
\end{lem}

Now we are in a position to prove Theorem \ref{th3}.

\vskip 2mm
\emph{Proof of Theorem \ref{th3}.}
Let $\{\mu_n\}\subset \mathcal{P}(\mathbb{R}^N)$ be a minimizing sequence of (\ref{eq 6}).
According to Lemma \ref{lemma 4} and the Prokhorov's theorem  \cite[Theorem 4.1]{B},
there exists a subsequence $\{\mu_{n_k}\} \subset\mathcal{P}(\mathbb{R}^N)$
and a measure $\mu_0\in \mathcal{P}(\mathbb{R}^N)$ satisfy
\begin{eqnarray*}
\mu_{n_k}\rightharpoonup \mu_0\ \ {\rm in} \ \  \mathcal{P}(\mathbb{R}^N).
\end {eqnarray*}
Clearly we get
\begin{eqnarray}\label{eq agh1}
E[\mu_0]\geq\inf \{E[\mu]:\rho\in \mathcal{D}_{M,m}\}=\lim_{n\rightarrow\infty}E[\mu_n].
\end{eqnarray}
Since $F(x)$ is continuous and bounded in $B(0,L)$, by  \cite[Definition 1.3.3]{B}
we have $$\int_{B(0,L)}F(x)d\mu_{n_k}(x)\rightarrow\int_{B(0,L)}F(x)d\mu_0(x) \ {\rm as} \ k\rightarrow\infty.$$
Similarly to (\ref{eq 9}), it follows that
\begin{eqnarray}\label{eq 11}
\int_{\mathbb{R}^N}F(x)d\mu_0(x)\leq\liminf_{k\rightarrow\infty}\int_{\mathbb{R}^N}F(x)d\mu_{n_k}(x)+\varepsilon.
\end {eqnarray}
We can deduce from  Lemma \ref{lemma 5} that
\begin{eqnarray}\label{eq 12}
\int_{\mathbb{R}^N} \int_{\mathbb{R}^N}K(x-y)d\mu_{n_k}(x)d\mu_{n_k}(y)\rightarrow\int_{\mathbb{R}^N}
\int_{\mathbb{R}^N}K(x-y)d\mu_0(x)d\mu_0(y),
\ {\rm \ as} \ k\rightarrow\infty.
\end {eqnarray}
In view of (\ref{eq 11}) and  (\ref{eq 12}), we see
\begin{eqnarray*}
&&\frac{1}{2}\int_{\mathbb{R}^N}\int_{\mathbb{R}^N}K(x-y)d\mu_0(x)d\mu_0(y)+\int_{\mathbb{R}^N}F(x)d\mu_0(x)
\\&&\leq\liminf_{k\rightarrow\infty}\left(\frac{1}{2}\int_{\mathbb{R}^N} \int_{\mathbb{R}^N}K(x-y)d\mu_{n_k}(x)d\mu_{n_k}(y)
+\int_{\mathbb{R}^N}F(x)d\mu_{n_k}(x)\right)+\varepsilon.
\end{eqnarray*}
Letting $\varepsilon\rightarrow0$ and using (\ref{eq agh1}) we obtain
$$\inf \{E[\mu]:\mu\in \mathcal{P}(\mathbb{R}^N)\}\leq E[\mu_0]\leq\liminf_{k\rightarrow\infty}E[\mu_{n_k}]=
\inf \{E[\mu]:\mu\in \mathcal{P}(\mathbb{R}^N)\},$$
which means that $\mu_0$ is a minimizer for (\ref{eq 6}) when $0<p<q$.\hfill$\Box$

\section{Proof of Theorem \ref{th4}}
This section is devoted to computing the global minimizer for the functional (\ref{eq 1})
with the endogenous potential satisfying
\begin{eqnarray}\label{123456}
K(x)=\frac{1}{2}|x|^2-\frac{1}{2-N}|x|^{2-N},\ \ N>2
\end{eqnarray}
and the exogenous potential
\begin{eqnarray}\label{1234561}
F(x)=\beta|x|^2,\ \ \ \beta>0.
\end{eqnarray}
We introduce the following lemma,
which determines the condition for $\rho_0\in \mathcal{D}_{M,m}$
to be a minimizer of the functional $E[\rho]$.

\begin{lem}\label{lemma 6}
Let $N\geq1$. Then $\rho_0\in \mathcal{D}_{M,m}$ is a local minimizer of the functional (\ref{eq 1})  if and  only if
\begin{eqnarray}\label{eq 17}
  \left\{\begin{array}{lcl}
     \medskip
       \psi(x)\geq c_0 \ \ {\rm a.e.\  on\  the \ set} \  \{x:\rho_0(x)=0\},\\
     \medskip
     \psi(x)=c_0 \ \ {\rm a.e.\  on\  the \ set} \  \{x:\rho_0(x)>0\},
  \end{array}\right.
\end{eqnarray}
where the function $\psi$ is defined by
\begin{eqnarray}\label{eq 13}
\psi(x)=\int_{\mathbb{R}^N}K(x-y)\rho_0(y)dy+F(x)
\end{eqnarray}
and $c_0:=\frac{\int_{\mathbb{R}^N}\rho_0(x)\psi(x)dx}{m}$ is a constant.
\end{lem}
\begin{proof}
First we prove the necessity.
The main idea of the proof is similar to the strategy introduced in \cite[Lemma 3.8]{CFT},
where the authors considered the functional with $F\equiv0$.
Since the exogenous potential $F$ is involved
in the computations, we prefer to give details for the convenience of the reader.

Let $\rho_0\in \mathcal{D}_{M,m}$ be a minimizer of the functional (\ref{eq 1}) and
$Z=\{\zeta\in L^1(\mathbb{R}^N)\cap L^\infty(\mathbb{R}^N):\zeta\geq 0,\
\int_{\mathbb{R}^N}\zeta(x)dx\leq\frac{m}{2}\}$.
For $0\leq\varepsilon\leq1$ and $\zeta\in Z$, we denote
$$\rho_\varepsilon(x)=\rho_0(x)+\varepsilon\left(\zeta(x)-\frac{\int_{\mathbb{R}^N}
\zeta(x)dx}{m}\rho_0(x)\right).$$
Noting that $\int_{\mathbb{R}^N}\rho_\varepsilon(x)dx=m$ and
for $0\leq\varepsilon\leq1$, $\zeta\geq 0$ and
$\int_{\mathbb{R}^N}\zeta(x)dx\leq\frac{m}{2}$, we have
\begin{eqnarray*}
\rho_\varepsilon(x)&=&\rho_0(x)\left(1-\varepsilon\frac{\int_{\mathbb{R}^N}
\zeta(x)dx}{m}\right)+\varepsilon\zeta(x)
\\&\geq&\frac{1}{2}\rho_0(x)+\varepsilon\zeta(x)\geq0.
\end{eqnarray*}
Define the function
\begin{eqnarray}\label{eq 114}
e[\varepsilon]:=E[\rho_\varepsilon]=E\left[\rho_{0}(x)+\varepsilon
\left(\zeta(x)-\frac{\int_{\mathbb{R}^{N}} \zeta(x) d x}{m} \rho_{0}(x)\right)\right]
\end{eqnarray}
on the interval $[0, 1]$. First note that $0$ is a boundary point, since for $ \varepsilon< 0$ and
$x\in \{x:\rho_0(x)=0\}\cap \{x:\zeta(x)>0\}$, we have $\rho_\varepsilon(x)=\varepsilon\zeta(x)<0$,
which means $\rho_\varepsilon(x)$ is not a member of the admissible class for $ \varepsilon< 0$.

Since
\begin{eqnarray}\label{eq 12345}
e[\varepsilon]\geq E[\rho_{0}] \ \ \ {\rm for} \ \  0\leq\varepsilon\leq1,
\end{eqnarray}
due to the local minimality of $\rho_0\in \mathcal{D}_{M,m}$, we deduce that
\begin{equation}\label{eq 14}
0\leq\lim_{\varepsilon\rightarrow0^+}\frac{e[\varepsilon]-e[0]}{\varepsilon-0}=
\left.\frac{de[\varepsilon]}{d\varepsilon}\right|_{\varepsilon=0^+}
=\left.\frac{dE[\rho_\varepsilon]}{d\varepsilon}\right|_{\varepsilon=0^+}.
\end{equation}
We compute for all $\zeta\in Z$
\begin{eqnarray}
\left.\frac{dE[\rho_\varepsilon]}{d\varepsilon}\right|_{\varepsilon=0^+}&=&\int_{\mathbb{R}^N}\int_{\mathbb{R}^N}K(x-y)\zeta(x)\rho_0(y)dxdy
+\int_{\mathbb{R}^N}F(x)\zeta(x)dx\nonumber
\\[1.5mm]&&-\frac{\int_{\mathbb{R}^N}\zeta(x)dx}{m}\left(\int_{\mathbb{R}^N} \int_{\mathbb{R}^N}K(x-y)\rho_0(x)\rho_0(y)dxdy+\int_{\mathbb{R}^N}F(x)\rho_0(x)dx\right)\nonumber
\\&=&\int_{\mathbb{R}^N}\left(\int_{\mathbb{R}^N}K(x-y)\rho_0(y)dy+F(x)\right)\zeta(x)dx\nonumber
\\[1.5mm]&&-\frac{\int_{\mathbb{R}^N}\left(\int_{\mathbb{R}^N}K(x-y)\rho_0(y)dy+F(x)\right)\rho_0(x)dx}{m}\int_{\mathbb{R}^N}\zeta(x)dx\nonumber
\\&=&\int_{\mathbb{R}^N}\psi(x)\zeta(x)dx-c_0\int_{\mathbb{R}^N}\zeta(x)dx\nonumber
\\&=&\int_{\mathbb{R}^N}(\psi(x)-c_0)\zeta(x)dx\geq0,\label{eq 1234}
\end{eqnarray}
where $\psi$ is defined in (\ref{eq 13}) and
$$c_0:=\frac{\int_{\mathbb{R}^N}\psi(x)\rho_0(x)dx}{m}.$$
From (\ref{eq 1234}), we can obtain
\begin{equation}\label{eq c1}
\psi(x)-c_0\geq0\ \  {\rm a.e. \ on\  }\  \mathbb{R}^N.
\end{equation}
Indeed, suppose that there exists a nonzero measure set $A\subset\mathbb{R}^N$ such that
$\psi(x)-c_0<0$ in $A$.
We choose
\begin{equation*}
\zeta(x) =\left\{ \arraycolsep=1pt
\begin{array}{lll}
 c_1, \quad
 \quad {\rm if}\   x\in A,\\[2mm]
 \phantom{    }
 0,\quad
 \quad \ {\rm if}\ x\in \mathbb{R}^N\backslash A
\end{array}
\right.\qquad
\end{equation*}
where $c_1>0$ and $c_1\cdot|A|\leq\frac{m}{2}$.
Clearly $\zeta\in Z$ and
$$\int_{\mathbb{R}^N}(\psi(x)-c_0)\zeta(x)dx=\int_{A}(\psi(x)-c_0)\zeta(x)dx+\int_{\mathbb{R}^N\backslash A}
(\psi(x)-c_0)\zeta(x)dx=c_1\int_{A}(\psi(x)-c_0)dx<0,$$
which contradicts with (\ref{eq 14}).
Furthermore, if there exists a nonzero measure set  $B\subset\{x:\rho_0(x)>0\}$ such that $\psi> c_0$ in $B$, then
$$c_0=\frac{\int_{\mathbb{R}^N}\psi(x)\rho_0(x)dx}{m}>\frac{c_0\int_{\mathbb{R}^N}\rho_0(x)dx}{m}= c_0,$$
which implies that
\begin{equation}\label{eq c2}
\psi(x)=c_0 \ \  {\rm a.e. \ on\  the \ set}\  \{x:\rho_0(x)>0\}.
\end{equation}
In view of (\ref{eq c1}) and (\ref{eq c2}),
we conclude that (\ref{eq 17}) is satisfied.

For the sufficiency, we only note that (\ref{eq 12345}) can be derived from (\ref{eq 17}) and (\ref{eq 1234}),
i.e.  $\rho_0\in \mathcal{D}_{M,m}$ is a local minimizer of the functional (\ref{eq 1}).
\end{proof}

\vskip 2mm

Now we  are ready to give the proof of Theorem \ref{th4}.

\vskip 2mm
\emph{Proof of Theorem \ref{th4}.}
Theorem \ref{th1} shows that the existence of a minimizer for (\ref{eq 4})
in $\mathcal{D}_{M,m}$.
Hence we only need to prove that $\rho_0(x)=\frac{m+2\beta}{\omega_N}\chi_{B(0,r_0)}(x)$
is the global minimizer of (\ref{eq 4}),
where $r_0=(\frac{m}{m+2\beta})^{1/N}$.
We separate the proof into three steps.

\medskip

\emph{Step 1.}  To prove $\rho_0\in \mathcal{D}_{M,m}$.
Noting that $\rho_0\geq0$ and $\|\rho_0\|_{L^\infty(\mathbb{R}^N)}=\frac{m+2\beta}{\omega_N}\leq M$,
where $\omega_N$ denotes volume of the unit ball in $\mathbb{R}^N$.
A simple calculation shows that
$$\int_{\mathbb{R}^N}\rho_0(x)dx=\int_{\mathbb{R}^N}\frac{m+2\beta}{\omega_N}
\chi_{B\left(0,(\frac{m}{m+2\beta})^{\frac{1}{N}}\right)}(x)dx
=\frac{m+2\beta}{\omega_N}{\omega_N}r_0^N=m,$$
$$ \int_{\mathbb{R}^N}\beta|x|^2\rho_0(x)dx<\beta r_0^2m$$
and
$$\int_{\mathbb{R}^N}(1+|x|)^2\rho_0(x)dx<(1+r_0)^2m,$$
which implies that $\rho_0\in \mathcal{D}_{M,m}$.

\vskip 2mm

\emph{Step 2.} To prove that $\rho_0(x)=\frac{m+2\beta}{\omega_N}
\chi_{B(0,r_0)}(x)$ is a local minimizer of the following functional
\begin{equation}\label{eq 15}
E[\rho]=\frac{1}{2}\int_{\mathbb{R}^N} \int_{\mathbb{R}^N}\left(\frac{1}{2}|x-y|^2
-\frac{1}{2-N}|x-y|^{2-N}\right)\rho(x)\rho(y)dxdy
+\int_{\mathbb{R}^N}\beta|x|^2\rho(x)dx.
\end{equation}
We need to prove that $\rho_0$ satisfies (\ref{eq 17}).
Let
\begin{equation}\label{eq c5}
\phi(x)=\int_{B\left(0,r_0\right)}\frac{1}{N(N-2)\omega_N|x-y|^{N-2}}dy,
\end{equation}
then $\phi$ is the solution of  the Poisson problem
\begin{equation*}
-\Delta\phi(x) =\left\{ \arraycolsep=1pt
\begin{array}{lll}
 \displaystyle  1,  \quad
 \quad {\rm if}\  |x|\leq r_0,\\[2mm]
 \phantom{    }
 \displaystyle  0,\quad
 \quad {\rm if}\  |x|>r_0.
\end{array}
\right.\qquad
\end{equation*}
Furthermore, $\phi(x)$ is radial and
$$-\Delta\phi(x)=-\frac{\partial^2\phi(r)}{\partial r^2}-\frac{N-1}{r}\frac{\partial\phi(r)}{\partial r}
=-\frac{1}{r^{N-1}}\frac{\partial}{\partial r}\left(r^{N-1}\frac{\partial\phi(r)}{\partial r}\right)$$
for $r=|x|$.
We integrate once to get
\begin{equation*}
\frac{\partial\phi(r)}{\partial r} =\left\{ \arraycolsep=1pt
\begin{array}{lll}
 \displaystyle  -\frac{r}{N},  \quad
 \quad  \phantom{zzzzz  zzzzz zzz}\  {\rm if}\  r\leq r_0,\\[2mm]
 \phantom{    }
 \displaystyle -\frac{m}{N(m+2\beta)r^{N-1}},\quad
 \quad \ {\rm if}\  r>r_0.
\end{array}
\right.\qquad
\end{equation*}
Using the fact that $\phi\in C^1$ and integrating once more, we achieve
\begin{equation}\label{eq c3}
\phi(r) =\left\{ \arraycolsep=1pt
\begin{array}{lll}
 \displaystyle  -\frac{r^2}{2N}+\frac{1}{2(N-2)}(\frac{m}{m+2\beta})^{\frac{2}{N}},  \quad
 \quad {\rm if}\  r\leq r_0,\\[2mm]
 \phantom{    }
 \displaystyle \frac{m}{(m+2\beta)N(N-2)r^{N-2}},\quad
 \qquad \ \ \ {\rm if}\  r>r_0.
\end{array}
\right.\qquad
\end{equation}

A simple computation yields,
\begin{eqnarray}
\psi(x)&=&\int_{\mathbb{R}^N}\left(\frac{1}{2}|x-y|^2-\frac{1}{2-N}|x-y|^{2-N}\right)\frac{m+2\beta}{\omega_N}
\chi_{B\left(0,r_0\right)}(y)dy+\beta|x|^2\nonumber
\\&=&\frac{m+2\beta}{\omega_N}\int_{B(0,r_0)}\frac{1}{2}|x|^2dy
+\frac{m+2\beta}{\omega_N}\int_{B(0,r_0)}\frac{1}{2}|y|^2dy\nonumber
\\&&+(m+2\beta)N\int_{B(0,r_0)}\frac{1}{N(N-2)
\omega_N|x-y|^{N-2}}dy+\beta|x|^2,\nonumber
\\&=&\frac{m+2\beta}{2\omega_N}\omega_N\frac{m}{m+2\beta}|x|^2+\beta|x|^2
+\frac{N(m+2\beta)}{2(N+2)}\left(\frac{m}{m+2\beta}\right)^{\frac{N+2}{N}}\nonumber
\\&&+(m+2\beta)N\int_{B(0,r_0)}\frac{1}{N(N-2)
\omega_N|x-y|^{N-2}}dy.\label{eq c7}
\end{eqnarray}
Putting (\ref{eq c5}) and (\ref{eq c3}) into
(\ref{eq c7}), we get
\begin{eqnarray*}
\psi(x)&=&\left\{ \arraycolsep=1pt
\begin{array}{lll}
 \displaystyle  \frac{N(m+2\beta)}{2(N+2)}\left(\frac{m}{m+2\beta}\right)^{\frac{N+2}{N}}
 +\frac{(m+2\beta)N}{2(N-2)}\left(\frac{m}{m+2\beta}\right)^{\frac{2}{N}}  \quad
 \ \  {\rm if}\  |x|\leq r_0,\\[2mm]
 \phantom{    }
\displaystyle \frac{m|x|^2}{2}+\frac{m}{N-2}|x|^{2-N}+\beta|x|^2+
\frac{N(m+2\beta)}{2(N+2)}\left(\frac{m}{m+2\beta}\right)^{\frac{N+2}{N}}
{\rm if}\  |x|>r_0.
\end{array}
\right.
\end{eqnarray*}
Since for $|x|>r_0$, $\psi(x)$ is an increasing function of $|x|$,
we obtain
$$\psi(x)\geq\psi(r_0)=\frac{N(m+2\beta)}{2(N+2)}
\left(\frac{m}{m+2\beta}\right)^{\frac{N+2}{N}}
 +\frac{(m+2\beta)N}{2(N-2)}\left(\frac{m}{m+2\beta}\right)^{\frac{2}{N}}=c_0,$$
for a constant $c_0$. It follows that
\begin{eqnarray*}
  \left\{\begin{array}{lcl}
     \medskip
       \psi(x)\geq c_0 ,\ \ \ \ \  {\rm if}\ |x|>\left(\frac{m}{m+2\beta}\right)^{\frac{1}{N}},\\
     \medskip
     \psi(x)=c_0, \ \ \ \ \  {\rm if}\ |x|\leq\left(\frac{m}{m+2\beta}\right)^{\frac{1}{N}}.
  \end{array}\right.
\end{eqnarray*}
Thus $\rho_0(x)=\frac{m+2\beta}{\omega_N}
\chi_{B(0,r_0)}(x)$ satisfies (\ref{eq 17}),
it can infer from Lemma \ref{lemma 6} that $\rho_0$ is a local minimizer of the functional (\ref{eq 15}).

\emph{Step 3.} To prove that
the functional (\ref{eq 15})
is strictly convex.
Let $x_0\in\mathbb{R}^N$ be the center of mass of density, that is
$\int_{\mathbb{R}^N}x\rho(x-x_0)dx=0$, so we can simplify the functional (\ref{eq 15})
\begin{eqnarray}
E[\rho]&=&\frac{1}{4}\int_{\mathbb{R}^N}\int_{\mathbb{R}^N}|x-y|^2 \rho(x)\rho(y)dxdy+\int_{\mathbb{R}^N}\beta|x|^2 \rho(x)dx-\frac{1}{2(2-N)}
\int_{\mathbb{R}^N}\int_{\mathbb{R}^N}\frac{\rho(x)\rho(y)}{|x-y|^{N-2}}dxdy\nonumber
\\&=&\frac{1}{4}\int_{\mathbb{R}^N}\int_{\mathbb{R}^N}|x-x_0-(y-x_0)|^2 \rho(x-x_0)\rho(y-x_0)dxdy+\int_{\mathbb{R}^N}\beta|x|^2 \rho(x)dx\nonumber
\\&&-\frac{1}{2(2-N)}\int_{\mathbb{R}^N}\int_{\mathbb{R}^N}\frac{\rho(x)\rho(y)}{|x-y|^{N-2}}dxdy\nonumber
\\&=&\frac{m}{2}\int_{\mathbb{R}^N}|x|^2\rho(x-x_0)dx+\int_{\mathbb{R}^N}\beta|x|^2 \rho(x)dx
+\frac{1}{2}N\omega_N\|\rho\|^2_{H^{-1}(\mathbb{R}^N)}\label{eq c4},
\end{eqnarray}
with
$$\|\rho\|^2_{H^{-1}(\mathbb{R}^N)}=
\int_{\mathbb{R}^N}\int_{\mathbb{R}^N}\frac{\rho(x)\rho(y)}{N(N-2)\omega_N|x-y|^{N-2}}dxdy.$$
Such $H^{-1}$-norm is given for instance in the proof of  \cite[Theorem 2.4]{CFT},
which is strictly convex.
Then for all $\rho_1$ and $\rho_2$ such that $\rho_1\neq\rho_2$ for $t\in (0,1)$, it follows that
\begin{eqnarray}\label{eq c8}
\|t\rho_1+(1-t)\rho_2\|^2_{H^{-1}(\mathbb{R}^N)}<t\|\rho_1\|^2_{H^{-1}(\mathbb{R}^N)}+(1-t)\|\rho_2\|^2_{H^{-1}(\mathbb{R}^N)}.
\end {eqnarray}
We also have
\begin{eqnarray} \label{eq c9}
\int_{\mathbb{R}^N}\beta|x|^2\left(t\rho_1(x)+(1-t)\rho_2(x)\right)dx
=t\int_{\mathbb{R}^N}\beta|x|^2\rho_1(x)dx+(1-t)\int_{\mathbb{R}^N}\beta|x|^2\rho_2(x)dx,
\end {eqnarray}
and
\begin{eqnarray}
&&\int_{\mathbb{R}^N}|x|^2\left(t\rho_1(x-x_0)+(1-t)\rho_2(x-x_0)\right)dx\nonumber
\\&&=t\int_{\mathbb{R}^N}|x|^2\rho_1(x-x_0)dx+(1-t)\int_{\mathbb{R}^N}|x|^2\rho_2(x-x_0)dx.\label{eq c10}
\end {eqnarray}
Substituting (\ref{eq c8})-(\ref{eq c10}) into (\ref{eq c4}), we can deduce that
$$E[t\rho_1+(1-t)\rho_2]<tE[\rho_1]+(1-t)E[\rho_2],$$
so $E[\rho]$ is strictly convex functional.

In conjunction with \emph{Steps 1-3}, we have
$\rho_0(x)=\frac{m+2\beta}{\omega_N}
\chi_{B(0,r_0)}(x)$ is the global minimizer of
the functional (\ref{eq 15}), which ends the proof.
\hfill$\Box$

\bigskip

\section*{\uppercase {Acknowledgments}}
The authors are supported in part by the National Natural Science Foundation of China (No. 11671079, No. 11701290,
No. 11601127 and No. 11171063), and the Natural Science Foundation of Jiangsu Province (No. BK20170896).

\end{document}